\documentclass[12pt,leqno,twoside]{amsart}

\usepackage[latin1]{inputenc}
\usepackage[T1]{fontenc}
\usepackage[colorlinks=true, pdfstartview=FitV, linkcolor=blue, citecolor=blue, urlcolor=blue]{hyperref}
\usepackage{amstext,amsmath,amscd, bezier,indentfirst,amsthm,amsgen,enumerate, geometry}
\usepackage[all,knot,arc,import,poly]{xy}
\usepackage{amsfonts,color, soul}  
\usepackage{amssymb}
\usepackage{latexsym}
\usepackage{epsfig}
\usepackage{graphicx}
\usepackage{srcltx}
\usepackage{enumitem}

\newtheorem{theorem}{Theorem}[section]

\newtheorem{lemma}[theorem]{Lemma}
\newtheorem{proposition}[theorem]{Proposition}
\theoremstyle{definition}
\newtheorem{definition}[theorem]{Definition}
\theoremstyle{remark}
\newtheorem{remark}[theorem]{\sc Remark}
\newtheorem{example}[theorem]{\sc Example}

\newtheorem{note}[theorem]{\sc Note}

\renewcommand{\Box}{\square}    

\renewcommand{\i}{{\rm{int}}}

\newcommand{\Sing}{{\mathrm{Sing\hspace{2pt}}}}

\newcommand{\rank}{{\mathrm{rank\hspace{2pt}}}}

\newcommand{\ord}{{\mathrm{ord}}}

\newcommand{\im}{{\mathrm{Im\hspace{2pt}}}}

\newcommand{\jac}{{\mathrm{Jac}}}

\newcommand{\e}{\varepsilon}
\newcommand{\m}{\setminus}
\newcommand{\s}{\subset}

\newcommand{\fin}{\hspace*{\fill}$\Box$\vspace*{2mm}}

\newcommand{\cO}{{\mathcal O}}

\newcommand{\bC}{{\mathbb C}}
\newcommand{\bP}{{\mathbb P}}

\begin{document}

\title[Local image of analytic map]{The local image problem for complex analytic maps}

\author{\sc Cezar Joi\c ta}
\address{Institute of Mathematics of the Romanian Academy, P.O. Box 1-764,
 014700 Bucharest, Romania and Laboratoire Europ\' een Associ\'e  CNRS Franco-Roumain Math-Mode}
\email{Cezar.Joita@imar.ro}

\author{Mihai Tib\u ar}
\address{Universit\' e de  Lille, CNRS, UMR 8524 -- Laboratoire Paul Painlev\'e, F-59000 Lille,
France}  
\email{mtibar@univ-lille.fr}



\thanks{The authors acknowledge the support of the Labex CEMPI
(ANR-11-LABX-0007-01). The first author acknowledges the
CNCS grant PN-III-P4-ID-PCE-2016-0330.}

\begin{abstract}
 We address the question ``when  the local image of a map is well defined'' and 
answer it in case of holomorphic map germs  with target $(\bC^{2}, 0)$. We prove a criterion for  holomorphic map germs $(X, x)\to (Y, y)$ to be locally open, solving a conjecture by Huckleberry  in all dimensions. 
\end{abstract}

\maketitle

\section{Introduction}

 Let  $F :(X,x) \to (Y,y)$ be a non-constant complex analytic map germ  between complex analytic set germs of pure dimensions. 
Consider an embedding $(X,x)\subset (\bC^N, 0)$ and the intersections of some representative $X$ with arbitrarily small balls $B_{\e}\s \bC^N$ centred at $0$.
  We say that $F$ has a well defined  \emph{local image} if and only if  the germ at $y$ of the set $F(B_{\e}\cap X)$ is independent of the  small enough $\e>0$.  This extends the notion of \emph{locally open maps}. The map germ $F :(X,x) \to (Y,y)$ is said  to  be \emph{open at $y$} (or that it is locally open, see e.g. \cite{Hu}) if for every open neighbourhood $U$ of $x$, $F(U)$ has $y$ as an interior point. Therefore ``locally open'' trivially implies ``well defined local image''.
  
  For instance in case of a  non-constant holomorphic function germ $f : (\bC^{n}, 0) \to (\bC, 0)$, the Open Mapping Theorem tells that $f(U)$ is open. Therefore the equality of set germs $(\im f, 0) = (\bC,0)$ holds,  thus  $f$ is locally open, and in particular the function germ $f$ has a well-defined image as a set-germ.
 
\medskip

If the target of a holomorphic map $F$ has higher dimension, then $F$ may still have a well defined local image, 
for instance if $F$ is a \emph{proper map at $y$}, by Remmert's Proper Mapping Theorem \cite{Re, GR1}.  Nevertheless,  without the properness, map germs as simple as $G: (\bC^{2},0) \to (\bC^{2},0)$, $(x,y) \mapsto (x, xy)$,  do not have a well defined local image.
In this example, the image by $G$  of any small ball $B_{\e}$ centred at the origin is a subanalytic set (but not analytic) which depends on the radius $\e$ so radically that the set germs $(G(B_{\e}), 0)$ and $(G(B_{\e'}), 0)$ are different  if $\e \not= \e'$.

\medskip

We address the following question:

 \medskip
\noindent
{\sc The image problem.}
\emph{Under what conditions the image of a complex analytic map germ is  a well defined set germ?}

 \medskip

We ask the local image to be a set germ, and not necessarily to be analytic. It is another long standing question under what conditions the local image of a holomorphic map is an analytic variety, see for instance \cite[p. 447]{Hu} for some comments.  If the image is a well defined set germ then it must be a subanalytic set germ (see Note \ref{n:note}). 
 
 Beyond the case of holomorphic functions germs $f : (\bC^{n}, 0) \to (\bC, 0)$ evoked before,  one encounters locally open images, according to Hamm's result \cite{Ha-icis}, also in case of complex map germs $F : (\bC^{n}, 0) \to (\bC^{p}, 0)$ which define an \emph{isolated complete intersection singularity} $(X,0) \subset (\bC^{n}, 0)$, where $X$ is the zero locus $Z(F)$ of $F$. Moreover, it turns out that the image by $F$ of the germ of the singular locus $(\Sing F,0)$ is an analytic set germ, and more precisely a hypersurface germ. It  then follows, cf  \cite{Ha-icis, Lo}, that there exists a fibration over the germ at 0 of the complement $\bC^{p}\m F(\Sing F)$, and on these bases one was able to further study the topology of the Milnor fibration  in relation to algebraic invariants\footnote{The well known monograph \cite{Lo}, recently re-edited, contains some of the most significant results on this rich topic.} of the isolated complete intersection $X$. 
 A far-reaching extension of local fibrations to  map germs $F$ such that their zero locus $Z(F)$  has \emph{nonisolated singularities} is done in \cite{JT}.


\medskip

 While the image problem remains widely open in general (and with almost no hope), we provide here the following  classification of holomorphic map germs $(f,g):(\bC^n,0)\to(\bC^2,0)$:

\begin{theorem}\label{t:main1}
Let $(f,g):(\bC^n,0)\to(\bC^2,0)$ be a non-constant holomorphic map germ. Then:

\begin{itemize}
\rm \item[(i)]  \it Its image is locally open,  i.e. $(\im (f,g), 0) = (\bC^{2}, 0)$,  if and only if:
  \begin{enumerate}
\rm \item \it   $\dim Z(f)\cap Z(g)=n-2$, or 
\rm \item \it  $\dim Z(f)\cap Z(g)=n-1$ and  $Z(f)\not\subset Z(g)$,  $Z(g)\not\subset Z(f)$, and 
$\im(f,g)$ is well defined as a set germ at $0$.
 \end{enumerate}
 
\rm \item[(ii)]  \it  In case $Z(g)\subset Z(f)$ or  $Z(f)\subset Z(g)$,  the map $(f,g)$  has a well defined local image if and only if $(\im (f,g), 0)$ is an irreducible plane curve germ, and this is equivalent to  $\jac(f,g)\equiv 0$.  
\end{itemize}
\end{theorem}
\medskip

 All situations described in Theorem \ref{t:main1} can be realised, and we discuss some examples in \S \ref{ex:notnice} and \S \ref{ex:nice1}.
The case (b) of Theorem \ref{t:main1} leaves us with the problem of how to decide whether the image of $(f,g)$ is well defined as a set germ (and thus it is locally open), which we will solve as follows. We first produce a handy test for the local openness of the image  in terms of \emph{gap lines},  Proposition \ref{p:crit}. Although this is a sufficient condition only, it leads to the notion of \emph{gap curves} which is central in our more general Proposition \ref{p:gapcurve} providing an equivalence criterion for a holomorphic map germ $F:(X,a)\to (Y,b)$ to be locally open, thus completing the solution to our problem.
 Moreover, this criterion allows us to find the answer to a much older problem, as follows.

\smallskip

About  50 years ago, Huckleberry addressed in \cite{Hu} the question of locally open maps, which is a particular case of the Image Problem stated above (see the discussion at the beginning of this Introduction). Huckleberry introduced the notion of a \emph{subflat} map (an algebraic condition, see Definition \ref{d:subflat}) and  proved that  a holomorphic map germ $(\bC^{2},0) \to (\bC^{2},0)$ is open if and only if it is subflat.
He conjectured  that this characterisation holds for any  holomorphic map germ $(X,x) \to (Y,y)$, in arbitrary dimensions. 

We derive from Proposition \ref{p:gapcurve} the following positive answer to Huckleberry's conjecture \cite{Hu}:

\begin{theorem}\label{t:huck} 
Let $F:(X,a)\to (Y,b)$,  $\dim X\geq \dim Y \ge 1$, be a  holomorphic map germ
between two germs of reduced locally irreducible complex spaces, and such that $\Sing F \not= X$. 

Then  $(\im F,b)=(Y,b)$ if and only if $F$ is subflat. 
\end{theorem}


\section{The image of the map germ $(f,g)$}\label{p:dim}

\subsection{Proof of Theorem \ref{t:main1}(i)(a)}\

Let $(H,0)\subset (\bC^{n},0)$ be a general complex $2$-plane  germ such that $0$ is  an isolated point of $H\cap (f,g)^{-1}(0)$. Then, by e.g. \cite[Proposition, page 63]{GR}, it follows that there exist
 an open ball  $B$ at $0$ in $\bC^{n}$  and   an open neighbourhood $U$ of the origin in $\bC^2$ 
such that $(f,g)(H\cap B)\subset U$ and the induced map $(f,g)_{|H\cap B}:H\cap B\to U$ is finite. 
By the Open Mapping Theorem  (cf \cite[ page 107]{GR}), this implies  that $(f,g)(H\cap B)$ is open, thus $(\im (f,g),0) = (\bC^{2}, 0)$.
\fin

\subsection{Proof of Theorem \ref{t:main1}(ii)}\label{ss:proofi}


\noindent
  If the image of $(f,g)$ is a curve (necessarily irreducible) then it is a germ by definition. In fact, if  $\im (f,g)$
is only included in a complex curve, and $\im (f,g)\not=\{(0,0)\}$, then  $\im (f,g)$ must be the whole irreducible curve germ, as an application of the Open Mapping Theorem.

\smallskip
\noindent
If $Z(g)\subset Z(f)$ then  $\im (f,g)$ cannot be $(\bC^{2},0)$ since the axis $\bC \times \{0\}$ is missing from the image.
 By the next result, if $\im (f,g)$ is a well defined set germ then it must be an irreducible curve germ.
\begin{proposition}\label{p:germ}
Let $(f,g):(\bC^n,0)\to(\bC^2,0)$ be non-constant holomorphic map germ. 

If  $\im (f,g)$ is a well defined set germ at the origin,  then either $(\im(f,g), 0)=(\bC^2,0)$ or $(\im(f,g),0)=(C, 0)$ where  $(C,0)\subset (\bC^2,0)$ is  an irreducible  complex curve germ.
\end{proposition}

To prove this proposition we need the following:
\begin{lemma}\label{l:union}
Suppose that $X$ is a complex space and $F:X\to \bC^2$ is a holomorphic map. Then $\im F$ can be written as a disjoint union $U\sqcup A$, where $U$ is a (possibly empty) open subset of $\bC^2$ and $A$ has 3-Hausdorff measure equal to 0.
\end{lemma}
\begin{proof}
We stratify $X$ by the rank of $F$ and write $X=M_2\sqcup M_1\sqcup M_0$, where the $M_j$'s are complex manifolds, such that $F_{|M_j}$ has rank $j$. Note that $M_j$ is not necessarily connected and its  connected components may not have the same dimension. Then $F(M_2)$ is open and $F(M_1\cup M_0)$ is a countable union
of complex curves or points and hence has 3-Hausdorff measure equal to 0. The subsets $U:=F(M_2)$ and $A:=F(X)\setminus U \subset F(M_1\cup M_0)$ verify our claim.
\end{proof}


\begin{proof}[Proof of Proposition \ref{p:germ}]
Let $B_\e\subset \bC^n$ be an open ball centred at the origin of radius $\e$ such 
that $(f,g)$ is holomorphic on $B_\e$. By our hypothesis,   
for all $0<\e'<\e$ we have the equality of germs $((f,g)(B_\e),0),0)=((f,g)(B_{\e'}),0)$. 

We fix $0<\e'<\e$ and $r>0$ such that $(f,g)(B_{\e'})\cap D_r=(f,g)(B_{\e})\cap D_r$ where $D_r\subset \bC^2$ denotes the open ball centred at 
$0$ and of radius $r$. Since 
$$(f,g)(B_{\e'})\cap D_r\subset (f,g)(\overline {B_{\e'}})\cap D_r\subset (f,g)(B_{\e})\cap D_r$$
we get the equality $(f,g)(\overline {B_{\e'}})\cap D_r= (f,g)(B_{\e})\cap D_r$.

By the above equality, since $K:=(f,g)(\overline {B_{\e'}})$  is a compact subset of $\bC^2$, the image  $(f,g)(B_{\e})\cap D_r=K\cap D_r$ is closed in $D_r$.

By Lemma \ref{l:union}  we have that $K\cap D_r$ is equal to a disjoint union $U\sqcup A$ of  an open subset $U\subset D_r$ and a subset
$A\subset D_r$  which has $3$-Hausdorff measure equal to zero.

 This implies:
\begin{equation}\label{eq:disjoint}
 D_r\setminus A=U\sqcup (D_r\setminus K)
\end{equation}
 which implies in turn that $A$ is closed in $D_r$. 
Since $A$ has $3$-Hausdorff measure equal to zero, it follows that $D_r\setminus A$ 
is connected, see e.g. \cite[Prop. 6, pag 347]{Ch}. 

We distinguish two cases,  where the notation $\i(A)$ is for the interior of the set $A$.

\smallskip
\noindent
\textbf{Case 1:} $0\in\overline{\i((f,g)(B_\e))}$.

Our assumption 
 implies that $0\in \partial \bar U$,  hence $U\neq\emptyset$. 
Then the disjoint union decomposition \eqref{eq:disjoint} of the connected set $D_r\setminus A$ into open sets shows that $D_r\setminus K=\emptyset$. Therefore we have the equality $K \cap D_r  = D_{r}$, which  shows that  the image $(f,g)(\overline {B_{\e'}})$ contains the ball $D_{r}$.  We conclude that $(\im(f,g),0)=(\bC^2,0)$ in this case.  

\smallskip
\noindent
\textbf{Case 2:} $0\not\in\overline{\i((f,g)(B_\e))}$. 

By shrinking $B_\e$ we may assume  that $\i(f,g)(B_\e)=\emptyset$.
This implies that all non-empty fibres $(f,g)^{-1}((f,g)(p))$ arbitrarily close to  $(f,g)^{-1}(0)$ have pure dimension $n-1$.
Indeed, if there is a sequence of points $x_{i}\to 0$ in the source, such that the local dimension of the fibre
 $(f,g)^{-1}((f,g)(x_{i}))$ is $n-2$ then we use the reasoning in the proof of Theorem \ref{t:main1}(a) which shows  
 $(f,g)(x_{i})$ is an interior point of the image for any $i$, and thus $0\in\overline{\i((f,g)(B_\e))}$, which contradicts the hypothesis.

In order to finish the proof we need the following result, which is also called  ``Remmert's Rank Theorem''  by \L ojasiewicz in  \cite[Theorem 1, pag 295]{Loj2}, see also \cite[Prop. 3, Ch. VII]{Na}: 
\begin{proposition}\label{p:nara}\cite[Satz 18, pag. 30]{Re0}
Let $X$ and $Y$ be complex spaces such that $X$ is  pure dimensional and $f:X\to Y$ be a 
holomorphic map. If $r = \dim_{x} f^{-1}(f(x))$ is independent of $x\in X$ then any point $a\in X$
has a fundamental system of neighbourhoods $\{U_{i}\}$ such that $f(U_{i})$ is analytic at $f(a)$, of dimension $\dim X - r$.
\fin
\end{proposition}

Since all non-empty fibres of  $(f,g)$ have dimension $n-1$, it  follows from Proposition \ref{p:nara} that
there exists a connected neighbourhood $W\subset \bC^n$ of $0$ such that $(f,g)(W)$ is a complex analytic subset of $\bC^2$, thus a complex curve $C$ containing the origin, by Proposition \ref{p:germ} and since in our case the image 
cannot be $(\bC^2, 0)$.
Since $W$ is connected, we get that $C$ is also irreducible, thus we get the equality $(\im(f,g),0)=(C,0)$.

This ends the proof of Proposition \ref{p:germ}, and thus of the first equivalence of Theorem \ref{t:main1}(ii). 
\end{proof}

\medskip
 Let us finally prove the second equivalence  claimed by  Theorem \ref{t:main1}(ii), namely:\\
  $(\im (f,g),0)$ is a well-defined curve germ $\Leftrightarrow$ $\jac(f,g) \equiv 0$. 

If $\im (f,g)$ is a curve then $\jac(f,g) \equiv 0$, trivially, so it remains to prove the converse.
The hypothesis $\jac(f,g) \equiv 0$ implies $\rank_{x}(f,g)\leq 1$, and since $\rank(f,g)\equiv 0$ implies $(f,g)\equiv 0$,  
we only have to deal with the case   $\rank(f,g)\not\equiv 0$. 

If $B$ denotes a small enough open ball centred at $0\in \bC^{n}$, then
$B_{0}:=\{x\in B  \mid \rank_x(f,g)>0\}$  is an open, connected and dense analytic subset of $B$, and actually
 $\rank_x(f,g)=1$, $\forall x\in B_{0}$. By the rank theorem  we deduce that
$\dim_x(f,g)^{-1}(f(x),g(x))=n -1$ for all $x\in B_{0}$.  The next result on the semi-continuity of the dimension of fibres is useful in order to figure out what happens at points in $B\m B_{0}$, see also \cite[page 66]{Na}: 

\begin{lemma}\cite[Satz 16]{Re0}
Let $F:X\to Y$ be a holomorphic map between complex spaces. Then every $a\in X$ has 
a neighbourhood $U\subset X$ such that 
\[ \dim_x F^{-1}(F(x))\leq \dim_a F^{-1}(F(a))\]
  for any $x\in U$.
  \fin
\end{lemma}

Since $B_{0}$ is dense in $B$,  the above lemma tells that in our case, for any $x\in B$ one has at least the inequality $\dim_x(f,g)^{-1}(f(x),g(x))\geq n -1$.  However, this inequality cannot be strict (i.e. not even at points in $B\m B_{0}$) since the converse inequality  $\dim_x(f,g)^{-1}(f(x),g(x))\leq n -1$ necessarily holds, as we have shown in the first part of the above  proof. 

And now,  since  $\dim_x(f,g)^{-1}(f(x),g(x))=n -1$ for all $x\in B$,  
our claim follows from Proposition \ref{p:nara} applied to the point $0\in B$.

This ends the proof of Theorem \ref{t:main1}(ii).
\fin

\subsection{Proof of Theorem \ref{t:main1}(i)(b)}\ 

If the image $\im(f,g)$ is a set germ then by Proposition \ref{p:germ} it is either a curve germ or $(\bC^2, 0)$. Let us show that it cannot be a curve. Indeed, if it is a curve $(C,0)$ different from the axes, then this has a Puiseux parametrisation, say  $(h(t), t^{\gamma})$, for some holomorphic function $h$ with $\ord_{0} h>0$, and some positive integer $\gamma$.  Then $g(x) =0$ implies that $t=0$, thus $f(x)=0$, which means $Z(g)\subset Z(f)$, and this contradicts the hypothesis. 
 If the curve $(C,0)$ is one of the axes,  then we immediately get the same contradiction (for instance, if the axis is $\bC \times \{0\}$, then this implies $Z(f)\subset Z(g)$). 
\fin 


\subsection{Examples where the image of $(f,g)$ is not a set germ}\label{ex:notnice}\label{ex:notnice2}\ 

\noindent
 Let  $G : \bC^{2} \to \bC^{2}$, $G(x, y) =  (x, xy)$. The global image of this algebraic map is the semi-algebraic set $(\bC^{2}\m \{ x=0\})\cup\{(0,0)\}$, 
  but since $\jac\hspace{1pt} G =x$ (thus not identically 0), Theorem \ref{t:main1}(ii) tells that the image of the map germ $G : (\bC^{2}, 0) \to (\bC^{2},0)$ is not well defined as a set germ.
    Actually, for  the 2-disks $D_t := \{|x|<t, |y|< t\}$ as basis of open neighbourhoods of 0 for $t>0$,
the image $A_t := G(D_t)$ is the open angle of vertex 0, having the horizontal axis as bisector,  and of slope $t$.  

\smallskip
\noindent
Let
 $F:\bC^2\to \bC^2$, $F(x,y)=(x(x+y), xy)$.  The global image of this algebraic map is $(\bC^{2}\m \Delta)\cup\{(0,0)\}$, where $\Delta$ denotes the diagonal line. 
 
Let us see that the image of the map germ $F$ is not a set germ  at $(0,0)$. 
The images of the line segments $x + (1-\alpha ) y=0$ inside a small ball $B_{\e}$ at the origin 
 are line segments centred at 0 in the target. When $\alpha$ tends to 1, the image segments tend to the diagonal $\Delta$, and their lengths tend to 0.  It follows that the images by $F$ of small balls intersected with arbitrarily small balls $D_{\delta}$ in the target are different, namely $F(B_{\e_{1}}) \cap D_{\delta} \not= F(B_{\e_{2}}) \cap D_{\delta}$.
  
 The property ``the image of $G$ is well defined as a set germ'' being invariant under change of coordinates in the source or in the target,  we consider the linear change of coordinates $(a, b) \mapsto (a-b, b)$ in the target. The resulting map germ
  $G :(\bC^2, 0)\to (\bC^2, 0)$, $G(x,y)=(x^{2}, xy)$  has one single \emph{gap line}  $\{ x=0\}$ (see \S \ref{ss:gap}),
and the image $\im G$ is not well defined as set germ, by Theorem \ref{t:main1}(ii), since $\jac F \not \equiv 0$.

\subsection{Examples where the image  of $(f,g)$ is a set germ}\label{ex:nice1}\label{ex:2var}\

\noindent
Let $F:(\bC^3, 0)\to (\bC^2, 0)$, $F(x,y,z) = (xy, xz)$. This satisfies the hypothesis of Theorem \ref{t:main1}(i)(b).
The fibers are not equidimensional, i.e. all fibres are curves except of the one over the origin which contains the plane $\{ x=0\}= \Sing F$. 
However, the image of any open ball $B_{\e} \s \bC^3$ centred at 0 contains the open ball $B_{\e^{2}} \s \bC^2$  centred at
0, thus the image of the map germ $F$ is a germ and $(\im F,0) =( \bC^2,0)$.

  Another example for Theorem \ref{t:main1}(i)(b) is  $F:(\bC^2, 0)\to (\bC^2, 0)$, $F(x,y)=(x(y+x^{2}),y(y+x^{2}))$. It is not trivial but still not difficult to show that $(\im F,0) =( \bC^2,0)$.  In \S \ref{ss:gap}  we give a sufficient criterion (Proposition \ref{p:crit}) for the image to be $(\bC^2, 0)$ which is easily verified by this example, cf  Example \ref{ex:im1}.


 
 \section{When is the image of a map germ locally open?}\label{s:Huckleb}
 
 We first give a sufficient test for the image of $(f,g)$ to be locally open which is handy (see Example \ref{ex:im1}), and show its limits.
Next we prove an ``if and only if''   criterium in full generality, that is for a holomorphic map germ $F:(X,a)\to (Y,b)$.

\subsection{Gap lines} \label{ss:gap}\

In order to test the ``locally open'' possibility in Theorem \ref{t:main1}(i)(b) 
we introduce the notion of ``gap line'', which loosely speaking means a line germ in  the target which does not contain points   of the image except of $0\in \bC^{2}$. 

\begin{definition}
 Suppose that  $\dim Z(f)\cap Z(g)=n-1$ and  $Z(f)\not\subset Z(g)$,  $Z(g)\not\subset Z(f)$.
 Let $f=h \hat f $, $g=h \hat g$,  where $h=\gcd(f,g) \in \mathfrak{m}_{n}$, up to invertible elements in $\mathcal{O}_n$, and such that $\hat f, \hat g \in \mathfrak{m}_{n}$.
  We say that $Z(\beta x + \alpha y) \subset \bC^{2}$  is a {\it gap line} for $(f,g)$
if   the analytic set germ   $(Z(\beta \hat f+\alpha \hat g), 0)$ is included in the fibre $(f,g)^{-1}(0,0)$.   
\end{definition}

\begin{remark}
  For a given map germ $(f,g)$ there are at most finitely many gap lines. Indeed, since $\hat f$ and $\hat g$ are co-prime, for a gap line we must  have the inclusion
 $Z(\beta \hat f+\alpha \hat g)\subset Z(h)$. As $Z(h)$ has finitely many irreducible components, our conclusion follows.
\end{remark}

\begin{proposition}\label{p:crit}
 Let  $f=h\hat f $, $g=h\hat g$,  where $h=\gcd(f,g) \in \frak{m}_{n}$, and such that $\hat f$ and $\hat g$ are not units.  
  If  $\dim Z(h)\cap Z(\beta \hat f+\alpha \hat g)=n-2$ for any  $\left[\alpha:\beta\right]\in\bP^1$, then $(\im(f,g), 0) = (\bC^{2}, 0)$.
\end{proposition}

The above result is not an equivalence, although it looks close to that, see Example \ref{ex:notequivgap}.  
The equality $(\im(f,g), 0) = (\bC^{2}, 0)$ implies that there are no gap lines (trivially), but the converse is not true, see Example \ref{ex:notequiv(b)}.

The proof of Proposition \ref{p:crit} will be given after the next \S \ref{ss:gapcurve} where we obtain a desired equivalence, and moreover in the most general setting.  But there is a price to pay.

\subsection{The gap variety, and a general  criterion for the existence of locally open image}\label{ss:gapcurve}

Let $F:(X,a)\to (Y,b)$,  $\dim X\geq \dim Y \ge 1$ be a  holomorphic map germ between two germs of reduced, locally irreducible complex spaces. 

\begin{definition}\label{d:gapvar}
Let  $(V,b)\subset (Y,b)$ be a complex analytic germ of positive dimension. We say that $(V,b)$ is a \emph{gap variety} for $F$ 
if  the inclusion of analytic set germs $(F^{-1}(V), a)\subset (F^{-1}(b), a)$ holds.
\end{definition}

 Let us remark that Huckleberry \cite{Hu} had used the same concept under the name  ``$F$ omits $V$''. 
Our general criterion is the following. 
 
\begin{proposition}\label{p:gapcurve}
  Let $F:(X,a)\to (Y,b)$,  $\dim X\geq \dim Y \ge 1$ be a  holomorphic map germ. Then
the image of $F$ is open at $b\in Y$ if and only if $F$ has no gap curve.
\end{proposition}

 Before giving the proof, let us remind a useful tool in the realm of subanalytic sets.
 
\begin{note}\label{n:note} 
The Curve Selection Lemma was first proved for semi-analytic sets, see e.g.  \cite{BC}, \cite{Mi}. We shall need it in the \emph{subanalytic} setting. The subanalytic concept has been introduced because the image by a real analytic map of a semi-analytic set is not semi-analytic in general. A subset $X$ of a real analytic manifold $M$ is called \emph{subanalytic} if for each point $x\in M$ there exists a neighbourhood $U\subset X$ of $x$, a real analytic manifold $N$, and a relatively
compact semi-analytic subset $A$ of $M\times N$ such that $U$ is the image of $A$ by the projection $M\times N\to M$. 
This category is closed under taking closures,  interiors,  complements, finite intersections, finite unions. Moreover, the image of a relatively compact subanalytic set by an  analytic map is subanalytic. For more details, see e.g. \cite{Ga}, \cite{Hi}, \cite{BM}, \cite{Loj3}.
 
  In the setting of subanalytic sets, the Curve Selection Lemma is due to  Hironaka \cite[Prop. 3.9, pag. 482]{Hi}:  \emph{Let $X\subset M$ be a subanalytic set and let $b\in \partial X$. There exists a real analytic function $\gamma: ]-\e,\e[ \to M$  such that $\gamma(0)=b$ and that $\gamma (]0,\e [)\subset X$. }
\end{note}

\subsection{Proof of Proposition \ref{p:gapcurve}.}
 ``$\Rightarrow$'' is obvious, so let us show ``$\Leftarrow$''. 
  Let us also recall from the Introduction that, by the very definitions,  if $F$ is locally open at $b\in Y$, then its image is well-defined as a set germ.
  
 If the image of $F$ is not open at $b\in Y$, then there exists some open ball $B$ 
centred at $a$ such that $b\in\partial(Y\setminus F(B))$, where $Y$ denotes here some representative of the analytic set germ $(Y,b)$.  Since $Y\setminus F(B)$ is a subanalytic set (as being the complement of the image of a relatively compact analytic set, see Note \ref{n:note} and its references), the above cited Curve Selection Lemma says that there exists a real 
analytic function $\gamma: ]-\e,\e[ \to Y$  such that $\gamma(0)=b$ and that $\gamma (]0,\e [)\subset Y\setminus F(B)$. We consider the complexification of  $\gamma(]-\e,\e[)$, i.e. the (unique) irreducible complex analytic
curve $C$ containing $\gamma(]-\e,\e[)$, and note that $C\subset Y$. The analytic set germ  $(F^{-1}(C), a)$ has finitely many  irreducible components;  let $(A,a)$ be one of the components such that $(A,a)\not\subset (F^{-1}(b), 0)$. Then the restriction $F_{|A}: (A,a) \to (C, b)$ is locally open, and this solves our problem.

For the reader's convenience, let us develop the argument why  $F_{|A}$ is locally open: by repeatedly slicing with generic hyperplanes, one may reduce to $\dim (A,a)=1$, and the generic slicing preserves the condition $F_{|A}$ is not the constant map. By the Riemann Extension Theorem $F_{|A}$ lifts to a holomorphic map between the normalisations $\tilde F_{|A} : (\tilde A,a) \to (\tilde C, \tilde b)$ which is open due to the Open Mapping Theorem.

\medskip

\subsection{Proof of Proposition \ref{p:crit}.}
  Let us suppose by contradiction that  the image of $(f,g)$ is not open at $(0,0) \in\bC^2$, or does not exist as a set germ.  
Applying Proposition \ref {p:gapcurve} to our setting, it follows that  $(f,g)$ has a gap curve germ $(C, 0)\subset(\bC^2,0)$. 
 If $\{\varphi=0\}$ is a local equation 
for $C$,  then  by Definition \ref{d:gapvar} we have the inclusion $Z(\varphi(f,g))\subset Z(f)\cap Z(g)=Z(h)\cup(Z(\hat f)\cap Z(\hat g))$.  Since $Z(\varphi(f,g))$ has pure dimension $n-1$ and since 
$\dim Z(\hat f)\cap Z(\hat g)=n-2$,  we deduce the inclusion $Z(\varphi(f,g))\subset Z(h)$.   
 In local coordinates we may write 
 $\varphi(x,y)=P(x,y)+\sum_{i+j\geq p+1}c_{i,j}x^iy^j$, where $P(x,y)$ is a homogeneous polynomial
of degree $p\geq 1$. 
Therefore  $\varphi(f,g))=P(\hat f,\hat g)h^p+h^{p+1} \tilde h$, for some holomorphic function $\tilde h$.
Then $Z(\varphi(f,g))\subset Z(h)$ implies that $Z(P(\hat f,\hat g)+h\tilde h)\subset Z(h)$. We then deduce
$Z(P(\hat f,\hat g)+h\tilde h)\subset Z(h)\cap Z(P(\hat f,\hat g))$, and in particular 
$\dim Z(h)\cap Z(P(\hat f,\hat g))=n-1$. 

Writing now the homogeneous polynomial $P(x,y)$ as a product of linear factors, we deduce that there exists
$\left[\alpha:\beta\right]\in\bP^1$ such that $\dim Z(h)\cap Z(\beta \hat f+\alpha \hat g)=n-1$.
This ends the proof of Proposition \ref{p:crit}.

\bigskip

The condition ``no gap curve'' is nice in theory but
not easily checkable in practice, thus proving the locally openness of $F$ usually amounts to direct computations. See Example \ref{ex:notequivgap} below,  preceded by a couple of other examples for the above described situations.

\begin{example}\label{ex:im1}
Let $F:(\bC^2, 0)\to (\bC^2, 0)$, $F(x,y)=(x(y+x^{2}),y(y+x^{2}))$, which we have already seen before (\S\ref{ex:2var}). 
It is trivial to check the criterion of Proposition \ref{p:crit}, thus we have $(\im F,0) =( \bC^2,0)$.  

Let us also remark that by changing coordinates locally  $x=u,  y=v-u^2$ one gets the map germ $(uv, v(v-u^2))$
which is an example that Huckleberry computed explicitely  in \cite[p. 449]{Hu}.
\end{example}

\begin{example}\label{ex:notequiv(b)}
Let $(f,g):\bC^2\to\bC^2$, $(f,g)(x,y)=(xy,x^2y^2+y^3)$. \\
Then $h(x,y)=y$, $\hat f(x,y)=x$, $\hat g(x,y)=x^2y+y^2$.\\
If $\left[\alpha:\beta\right]\neq[1:0]$ then $Z(\beta \hat f+\alpha \hat g)\cap Z(h)=\{(0,0)\}$ and hence 
$Z(\beta \hat f+\alpha \hat g)\not\subset Z(h)$.\\
If $\left[\alpha:\beta\right] =[1:0]$ then $Z(\beta \hat f+\alpha \hat g)=Z(y(x^2+y))\not\subset Z(y)$.

We deduce that  $(f,g)$ has no gap line. Nevertheless $(f,g)$ has a gap curve since $\im f\cap\{(u,v)\in\bC^2 \mid v=u^2\}=(0,0)$. By Proposition \ref{p:gapcurve} together with Theorem \ref{t:main1}(i)(b), it then follows that the image of $(f,g)$ is not a well defined set germ.
\end{example}

\begin{example}\label{ex:notequivgap}
Let $(f,g):\bC^2\to\bC^2$, $(f,g)(x,y)=(x(x^4+y),y(x^4+y)^2)$.  Note that $h(x,y)=x^4+y$, $\hat f(x,y)=x$, $\hat g(x,y)=y(x^4+y)$
and hence $\dim Z(h)\cap Z(\hat g)=n-1$.
We claim that $(\im (f,g),0)=(\bC^2,0)$. 

Let $D_r\subset \bC$ denote the closed disk centred at the origin and of radius $r$.
We will show that for any $\frac  12 > \e>0$, there is  $r>0$ such that   $f(D_\e\times D_\e)\supset D_r\times D_r$. 
We shall actually prove this inclusion in the following by assuming that $r<\e^{10}$.

We need to show that for any $(a,b)\in D_r\times D_r$ there exists $(x_0,y_0)\in D_\e\times D_\e$
such that $(f,g)(x_0,y_0)=(a,b)$. The proof falls into two cases.
\medskip

\noindent
Case 1. $|a|^2\ge |b|$. Let $k:=\frac b{a^2}$, thus $|k|\le 1$. We consider the equation $x(x^4+k^2x^2)=a$. This  has five complex solutions and their product is $a$. It follows that for 
at least one of them, say $x_0$, we have $|x_0|\le |a|^{1/5}<r^{1/5}<\e$. For $y_0:=kx_0^2$ we get
$(x_0,y_0)\in D_\e\times D_\e$ and $f(x_0,y_0)=(a,b)$.
\medskip

\noindent
Case 2. $|a|^2\le |b|$. Let $k' :=\frac {a^2}b$, thus $|k'|\le 1$. 
We consider the equation $y((k')^2y^2+y)^2=b$ and  claim that this equation has at least three solutions in $D_{\e^2}\subset D_\e$. 
Since the  three solutions of the equation $y^3-b=0$ 
are in $D_{\e^2}$, by Rouche's Theorem it suffices to show that $|(k')^4y^5+2(k')^2y^4|<|y^3-b|$ on $\partial D_{\e^2}$. 
However, if $|y|=\e^2$, we have
$|y^3-b|\ge \e^6-r > \e^6-\e^{10}\ge \frac{15}{16}\e^6$ (since $\e<1/2$). On the other hand
$|(k')^4y^5+2(k')^2y^4|<\e^{10}+2\e^8<(\frac 1{16} +\frac 12)\e^6<\frac {15}{16}\e^6$, and therefore we have 
$|(k')^4y^5+2(k')^2y^4|<|y^3-b|$ on $\partial D_{\e^2}$ indeed.

We then use such a solution $y_0$. For  $x_0 :=\frac{a}{(k')^2y_0^2+y_0}$  we get
$x_0^2=\frac{a^2}{b}y_0=k'y_0$ and therefore $|x_0|^2<\e^2$, thus $|x_0|<\e$. We also have 
$x_0(x_0^4+y_0)=a$ and $y_0(x_0^4+y_0)^2=b$, thus  $(f,g)(x_0,y_0)=(a,b)$.
 
\end{example}

\subsection{Locally open image and the ``subflat'' condition}\label{p:ndim}


Huckleberry conjectured in \cite[p. 461]{Hu} that if $F$ is not totally singular (i.e. $\Sing F = X$ as germs at $a$), then
$(\im F,a)=(Y,b)$  if and only if  $F$ is  \emph{subflat}. 

\begin{definition}\cite{Hu}\label{d:subflat}
If $X$ and $Y$ are reduced, locally irreducible complex spaces and $F:X\to Y$ is a holomorphic mapping then $F$ is called \emph{subflat}
at $p\in X$ if  for every prime ideal $I\subset \cO_{Y,F(p)}$ such that $\dim V(I)>0$ we have that $\langle F^*(I)\rangle\cap F^*(\cO_{Y,F(p)})=F^*(I)$, where $\langle F^*(I)\rangle$ denotes the ideal generated by $F^*(I)$.
\end{definition}

Huckleberry proves his conjecture in case of holomorphic maps  $(\bC^{2},0) \to (\bC^{2},0)$.
We are now in position to derive a proof in the general setting:

\begin{theorem}\label{t:huck2} 
Let $F:(X,a)\to (Y,b)$,  $\dim X\geq \dim Y \ge 1$, be a  holomorphic map germ\footnote{Again,  this only makes sense for $F$ which are not totally singular, i.e. $\Sing F \not= X$.}
between two germs of reduced, locally irreducible complex spaces. Then  $(\im F,b)=(Y,b)$ if and only if $F$ is subflat. 
\end{theorem}
\begin{proof}
 Huckleberry \cite[Prop. 3.2]{Hu} had  actually proved in the same setting  the following statement, by using essentially the Nullstellensatz: \\
 
 \noindent
 (*) \emph{a  holomorphic $F$ such that $\Sing F \not= X$  is  subflat if and only if $F$ has no gap curve}.  \\
 
 The notion of ``gap curve '' has been given in Definition \ref{d:gapvar}. Therefore the proof of Theorem \ref{t:huck2} reduces, via (*),  to the equivalence ``\emph{$(\im F,b)=(Y,b)$ iff  $F$ has no gap curve}''  which is precisely our Proposition \ref{p:gapcurve}.
\end{proof}


\begin{remark}
Given two germs of complex spaces $(X,a)$ and $(Y,b)$, it might happen that there is no holomorphic map germ $F:(X,a)\to (Y,b)$ with $(\im F,b)=(Y,b)$.
 In \cite{CJ} one obtains a characterisation of all two-dimensional  complex germs $(Y,b)$ for which there exists a holomorphic map $F:(\bC^n,0)\to (Y,b)$ such that $(\im F,b)=(Y,b)$.
\end{remark}






\vspace{\fill}
\end{document}